\documentclass[11pt]{amsart}

\usepackage{amsfonts,amssymb,graphicx,epsfig,color,fancyhdr,ifthen,mathtools}
\usepackage[letterpaper,asymmetric,left=1.2in,right=1.2in,top=1.25in,bottom=1.
25in,bindingoffset=0.0in]{geometry}

\usepackage[small,nohug,heads=vee]{diagrams}
\usepackage{xy}

\theoremstyle{plain}
\newtheorem{thm}{Theorem}[section]
\newtheorem*{thm*}{Theorem}
\newtheorem{prop}[thm]{Proposition}
\newtheorem{cor}[thm]{Corollary}

\theoremstyle{definition}
\newtheorem{defn}[thm]{Definition}
\newtheorem{rmk}[thm]{Remark}
\newtheorem{example}[thm]{Example}

\newcommand{\loc}{{\rm loc}}
\newcommand{\CP}{\mathbb{CP}}
\newcommand{\C}{\ensuremath{\mathcal{C}}}

\newcommand{\sm}{\setminus}
\newcommand{\ra}{\rightarrow}

\newcommand{\supp}{{\rm supp}}
\newcommand{\pol}{{\rm pol}}

\newcommand{\field}[1]{\mathbb{#1}}
\newcommand{\CC}{\ensuremath{\field{C}}} 
\newcommand{\RR}{\ensuremath{\field{R}}} 
\newcommand{\NN}{\ensuremath{\field{N}}} 
\newcommand{\QQ}{\ensuremath{\field{Q}}} 
\newcommand{\DD}{\ensuremath{\field{D}}} 
\newcommand{\ZZ}{\ensuremath{\field{Z}}}

\newcommand{\eps}{\epsilon}

\newcommand{\comm}[1]{}

\newcommand{\showcomments}{no}

\newsavebox{\commentbox}
\newenvironment{comment}%
{\ifthenelse{\equal{\showcomments}{yes}}%
{\footnotemark
    \begin{lrbox}{\commentbox}
    \begin{minipage}[t]{1in}\raggedright\sffamily\small
    \footnotemark[\arabic{footnote}]}
{\begin{lrbox}{\commentbox}}}%
{\ifthenelse{\equal{\showcomments}{yes}}%
{\end{minipage}\end{lrbox}\marginpar{\usebox{\commentbox}}}
{\end{lrbox}}}

\begin{document}

\title{A dichotomy for Fatou components\\ of polynomial skew products}

\author[R.K.W. ~Roeder]{Roland K.\ W.\ Roeder$^1$}
\address{IUPUI Department of Mathematical Sciences\\
LD Building, Room 270\\
402 North Blackford Street\\
Indianapolis, Indiana 46202-3216\\ USA}
\email{rroeder@math.iupui.edu}

\date{\today}

\begin{abstract}
We consider polynomial maps of the form $f(z,w) = (p(z),q(z,w))$ that extend as
holomorphic maps of $\CP^2$.  Mattias Jonsson introduces in \cite{JON_SKEW} a
notion of connectedness for such polynomial skew products that is analogous to
connectivity for the Julia set of a polynomial map in one-variable.  We prove
the following dichotomy: if $f$ is an Axiom-A polynomial skew
product, and $f$ is connected, then every Fatou component of $f$ is
homeomorphic to an open ball; otherwise, some Fatou component of $F$ has
infinitely generated first homology.
\end{abstract}



\maketitle

\markboth{\textsc{R.K.W. Roeder}}
  {\textit{Dichotomy for polynomial skew products
  }}

\footnotetext[1]{Research was supported in part by startup funds from the Department of Mathematics at IUPUI.}

%========================================
\section{Introduction}
%========================================

We consider the dynamics of mappings $f:\CC^2 \ra \CC^2$ of the form
\begin{eqnarray}
f(z,w) = (p(z), q(z,w))
\end{eqnarray}
\noindent
where $p$ and $q$ are polynomials.  It will be convenient to assume that ${\rm
deg}(p) = {\rm deg} (q) = d$ and $p(z) = z^d + O(z^{d-1})$ and $q(z,w) = w^d
+O_z(w^{d-1})$ so that $f$ extends as a holomorphic mapping $f:\CP^2 \ra
\CP^2$.   (Throughout this paper we will assume that $d \geq 2$.) Since $f$
preserves the family of vertical lines $\{z \} \times \CC$, one can analyze $f$
via the collection of one variable fiber maps $q_z(w) = q(z,w)$, for each $z
\in \CC$.  In particular, one can define fiber-wise filled Julia sets $K_z$ and
Julia sets $J_z$ that are analogous to their one-dimensional counterparts.  For
this reason, polynomial skew products provide an accessible generalization of
one variable dynamics to two variables.  They been previously studied by many
authors, including Heinemann in \cite{HEIN1,HEIN2}, Jonsson in \cite{JON_SKEW},
DeMarco-Hruska in \cite{S}, and Hruska together with the author of this note in~\cite{HRUSKA_ROEDER}.  

A more general situation in which the base is allowed to
be an arbitrary compact topological space, while the vertical fibers are copies
of $\CC$, has been considered by Sester in \cite{SESTER1,SESTER2}.  Meanwhile,
generalization to semigroups of polynomial (and rational) mappings of the
Riemann sphere has been studied by extensively by Hinkkanen, Martin, Ren,  Stankewitz, Sumi, Urba{\'n}ski, 
and many others---we refer the reader to the excellent bibliography from
\cite{SUMI_III} for further references.

As an analogy with polynomial maps of one variable, 
\begin{defn}{\bf (Jonsson)} \label{DEFN}
A polynomial skew product $f$ is {\em connected} if $J_p$ is connected and $J_z$ is connected for all $z \in J_p$.
\end{defn}
\noindent
Here, $J_p$ is the Julia set for the ``base map'' $z \mapsto p(z)$.
Jonsson proves in \cite[Sec. 6]{JON_SKEW} that if $f$ is a connected polynomial
skew product, then the support $J_2 := \supp \mu$ of its measure of maximal entropy $\mu$ (see \S
\ref{SEC:POLY_SKEW}, below) is connected, the Hausdorff dimension of $\mu$ is
$2$, and the associated Lyapunov exponents are precisely $\log d$.  (Note that
$J_2$ can be connected even if $f$ is a disconnected polynomial skew product; see
\cite[Lemma 5.5]{S}.)

The Fatou set $U(f)$ of a holomorphic map $f: \CP^2 \ra \CP^2$ is the maximal
open set in $\CP^2$ on which the sequence of iterates is normal and the Julia set is its complement:
$J(f) := \CP^2 \sm U(f)$.  In
this note, we prove
\begin{thm}\label{THM}
Suppose that $f$ is an Axiom-A polynomial skew product.  If $f$ is connected, then every Fatou component of $f$ is homeomorphic to an open ball.  Otherwise, some Fatou component of $f$ has infinitely generated first homology.
\end{thm}
\noindent
The Axiom-A assumption ensures hyperbolicity of $f$ on the non-wandering set.
See, for example, \cite[Sec. 8]{JON_SKEW} and \cite{S}.  
Using the equivalence between (1)
and (2) from Theorem \ref{THM:JON} (below), connectivity of $f$ can be
expressed in terms of ``escape of the critical locus'', so Theorem \ref{THM} is
an analog of the fundamental dichotomy from polynomial dynamics in one variable
(see, for example \cite[Theorem 9.5]{MILNOR}).

\begin{cor}\label{COR} 
If $f$ is an Axiom-A polynomial skew product for which $J_2 = \supp \mu$ is disconnected,
then some Fatou component of $f$ has infinitely generated first homology.
\end{cor}

We will discuss in Remark \ref{RMK:HYPERBOLICITY} the natural question of
whether Theorem \ref{THM} holds if $f$ is not Axiom-A.
We first show that some hypothesis of hyperbolicity is needed, by providing an
example of a disconnected product mapping that is not Axiom-A, whose Fatou set
has a single component that is homeomorphic to an open ball.  We will then show
that the proof of Theorem \ref{THM} that is presented in Section \ref{SEC:MAIN}
continues to hold if we replace Axiom-A with slightly weaker conditions.

\subsection*{Acknowledgments}
Many of the techniques and ideas from this paper were originally inspired by
John H. Hubbard during discussions that we had when I was his student (5 years ago).  I have
also benefited greatly from discussions with Laura DeMarco, Mattias Jonsson,
Suzanne Hruska, Lex Oversteegen, and Rodrigo Perez.  

The anonymous referee has provided many helpful comments leading to substantial improvements in the paper.

\section{Background on polynomial skew products}
\label{SEC:POLY_SKEW}
%========================================

We present some background on polynomial endomorphisms of $\CP^2$ from
\cite{FS2,HP2,UEDA,SI1,BEDFORD_JONSSON} and on polynomial skew products from
\cite{JON_SKEW}. 

Let $f: \CC^2 \ra \CC^2$ be given by
\begin{eqnarray*}
f(z,w) = (p(z,w),q(z,w)),
\end{eqnarray*}
with $p$ and $q$ polynomials of degree $d \geq 2$, and let $p_d$ and $q_d$ be
the homogeneous parts of $p$ and 
$q$, respectively, of the  maximal degree $d$.  
Such a mapping extends as a holomorphic mapping $f: \CP^2 \ra
\CP^2$ if and only if $p_d(z,w) = 0 = q_d(z,w)$ implies that $(z,w) = (0,0)$.
Any holomorphic mapping $f: \CP^2 \ra \CP^2$ that is obtained in this way is
called a {\em polynomial endomorphism} of $\CP^2$.

Suppose that $f: \CC^2 \ra \CC^2$ extends as a polynomial endomorphism of $\CP^2$.
The (affine) Green's function
\begin{eqnarray}\label{EQN:LOCAL_GREEN}
G(z,w) = \lim \frac{1}{d^n} \log_+ ||f^n(z,w)||, \,\,\, \mbox{where} \,\,\, \log_+ :=
\max\{\log,0\},
\end{eqnarray}
\noindent
is a plurisubharmonic (PSH) function having the
property that $(z,w) \in \CC^2$ is in $U(f)$ if and only if $G(u,v)$ is
pluriharmonic (PH) in a neighborhood of $(z,w)$.  
The {\em Green's current} $T:=\frac{1}{2\pi}dd^c
G$ is a closed positive $(1,1)$ current on $\CC^2$ with
$J(f):= \supp T$ and $U(f):= \CC^2 \sm \supp T$.

If $f$ is a polynomial skew product, i.e. $p(z,w) \equiv p(z)$, there is a
refinement of this description.  The base map $p(z)$ has a Julia set $J_p
\subset \mathbb{C}$ and, similarly, a Green's function $G_p(z):= \lim_{n \to
\infty} \frac{1}{d^n} \log_+ ||p^n(z)||$.  Furthermore, one can define a
fiber-wise Green's function by:
\begin{eqnarray*}
G_z(w) := G(z,w) - G_p(z).
\end{eqnarray*}
\noindent
For each fixed $z$, $G_z(w)$ is a subharmonic function of $w$  and one defines
the fiber-wise Julia sets by $K_z := \{G_z(w) = 0\}$ and $J_z := \partial K_z$.

The extension of $f$ to $\CP^2$ is given by
\begin{eqnarray} \label{EQN:SKEW_HOMOG}
f([Z:W:T]) = [P(Z,T):Q(Z,W,T):T^d],
\end{eqnarray} \noindent
where $P(Z,T)$ and $Q(Z,W,T)$ are the homogeneous versions of $p$ and $q$.
The point $[0:1:0]$ that is ``vertically at infinity'' with respect to the affine coordinates $(z,w)$ is a totally invariant super-attracting fixed point and $(z,w) \in W^s([0:1:0])$ if and only if
$w \in \CC \setminus K_z$.

The action of $f$ on the line at infinity $\Pi:=\{T=0\}$ is given by the
polynomial map $f_\Pi(u) = q_d(1,u)$, where $u = w/z$ and
$q_d(z,w)$ is the homogeneous part of $q$ of maximal degree $d$.  As usual, one
can consider the associated Julia sets $K_\Pi$ and $J_\Pi$.

One can extend the Green's current $T$ as a closed
positive $(1,1)$ current on all of $\CP^2$ satisfying $f^* T = d \cdot T$ and
having the property that $J(f) = \supp T$ and $U(f)~=~\CP^2~\sm~\supp~T$.
Moreover, the wedge product $\mu:=T \wedge T$ is a measure satisfying $f^* \mu
= d^2 \cdot \mu$, which happens to be the measure of maximal entropy for $f$.
It is customary to define a second Julia set $J_2 \equiv J_2(f) := \supp \mu
\subset J(f)$.

If $f=(p,q)$ is a polynomial skew product, then \cite[Cor. 4.4]{JON_SKEW} gives
\begin{eqnarray*}
J_2 = \overline{\bigcup_{z \in J_p} \{z\} \times J_z}.
\end{eqnarray*} 

The following is (an excerpt from) Theorem 6.5 from \cite{JON_SKEW}:
\begin{thm}{\bf (Jonsson)}\label{THM:JON}
The following are equivalent:
\begin{enumerate}
\item $f$ is connected.
\item $C_p \subset K_p$ and $C_z \subset K_z$, for all $z \in J_p$.
\item $J_p$ is connected, $J_\Pi$ is connected, and $J_z$ is connected, for all $z \in \CC$.
\item $C_p \subset K_p$, $C_\Pi \subset K_\Pi$, and $C_z \subset K_z$, for all $z \in \CC$.
\end{enumerate}
\noindent
Here $C_p$ and $C_\Pi$ are the critical points of $p$ and $f_\Pi$, while $C_z := \{(z,w) : \partial q/\partial w = 0\}$.
\end{thm}

The following appears in \cite{HRUSKA_ROEDER}:
\begin{thm}{(\bf Hruska-R)}\label{THM:HR}
If $J_z$ is connected for every $z \in J_p$, then $W^s([0:1:0])$ is homeomorphic to an open ball.
Otherwise, the first homology $H_1(W^s([0:1:0]))$ is infinitely generated.
\end{thm}

\noindent
The proof that $W^s([0:1:0])$ is homeomorphic to a ball is an application of $G_z(w)$ as a type of Morse function, while the latter
uses non-trivial linking numbers between closed loops in $W^s([0:1:0])$ and the Green's current $T$.

%========================================
\section{Background on linking with a closed positive $(1,1)$ current in $\CP^2$.}
\label{SEC:LINKING}
%========================================

We present a brief summary (without proofs) of material from \cite[\S 3]{HRUSKA_ROEDER}.
See also \cite{ROE_NEWTON}.

Any closed positive $(1,1)$ current $S$ on a complex manifold $N$ can be
described using an open cover $\{U_i\}$ of $N$ together with PSH functions
$v_i: U_i \ra [-\infty,\infty)$ that are  chosen so that $S=dd^c v_i$ in each
$U_i$.  The functions $v_i$ are called {\em local potentials} for $S$ and they
are required to  satisfy the compatibility condition that $v_i - v_j$ is PH on
any non-empty intersection $U_i \cap U_j \neq \emptyset$.  The {\em support} of
$S$ and {\em polar locus} of $S$ are defined by:
\begin{eqnarray*}
\supp S &:=& \{z \in N \, : \, \mbox{if $z \in U_j$ then $v_j$ is not PH at $z$}\}, \,\, \mbox{and} \\
\pol S &:=& \{z \in N \, : \, \mbox{if $z \in U_j$ then $v_j(z) = -\infty$}\}.
\end{eqnarray*}
The compatibility condition assures that that above sets are well-defined.  Moreover,
since PH functions are never $-\infty$, we have $\pol S \subset \supp S$.

Let $M$ be another complex manifold, possibly of dimension different from that
of $N$.  If $f:M \rightarrow N$ is a holomorphic map with $f(M) \not \subset
\pol S$, then the pull-back $f^*S$ is a closed positive $(1,1)$ current defined
on $M$ by pulling back the system of local potentials for $S$ to form a system
of local potentials on $M$ that define $f^*S$.  See \cite[Appendix A.7]{SI1}
and \cite[p.  330-331]{HP2} for further details.

Given any closed positive $(1,1)$ current $S$ on $N$ and any piecewise smooth two chain $\sigma$ in $N$ with $\partial \sigma$ disjoint from $\supp \ S$,
we can define
\begin{eqnarray*}
\left<
\sigma,S \right> = \int_\sigma \eta_S,
\end{eqnarray*}
\noindent
where $\eta_S$ is a smooth approximation of $S$ within it's cohomology class in
$N-\partial \sigma$, see \cite[pages 382-385]{GH}. The resulting number $\left< \sigma,S
\right>$ will depend only on the cohomology class of $S$ and the homology class
of $\sigma$ within $H_2(N,\partial \sigma)$.  If $\sigma$ is
a holomorphic chain, this pairing simplifies to be
the integral of the measure\footnote{This measure is well-defined, since $\partial \sigma$ is
disjoint from $\supp S \supset \pol S$.} $\sigma^ * S$ over $\sigma$.  

The following invariance property is useful:
\begin{prop}\label{PROP:PAIRING_INV}
Suppose that $S$ is a closed positive $(1,1)$ current on $N$ and $f:M
\rightarrow N$ is holomorphic, with $f(M) \not \subset \pol S$.
If $\sigma$ is a piecewise smooth two chain in $M$ with $\partial \sigma$
disjoint from $\supp f^* S$, then
$\left< f_* \sigma, S \right> = \left< \sigma, f^*S \right>$.
\end{prop}

Notice that $H_2(\CP^2)$ is generated by the class of any complex projective line
$L \subset \CP^2$.  Since $S$ is non-trivial, $\left<L,S\right> \neq 0$, so that
after an appropriate rescaling we can assume that $\left<L,S\right> =1$. 
This normalization is satisfied by
the Green's Current $T$ that was defined in \S \ref{SEC:POLY_SKEW}.

\begin{defn}\label{DEFN:LK}
Let $S$ be a normalized closed positive $(1,1)$ current on $\CP^2$ and let $\gamma$
be a piecewise smooth closed curve in $\CP^2 \setminus \supp(S)$.
We define the {\em linking number}
$lk(\gamma,S)$ by
\begin{eqnarray*}
lk(\gamma,S) := \left< \Gamma,S \right> \ (\text{mod} \ 1)
\end{eqnarray*}
\noindent
where $\Gamma$ is any piecewise smooth two chain with $\partial \Gamma =
\gamma$.
\end{defn}

Unlike linking numbers between closed loops in $\mathbb{S}^3$, it is often the
case that that $\left< \Gamma,S \right> \not \in \ZZ$, resulting in non-zero
linking numbers $(\text{mod} \ 1)$.  

\begin{prop}\label{PROP:LINKING_DEPENDS_ON_HOMOLOGY}
If $\gamma_1$ and $\gamma_2$ are homologous in $H_1(\CP^2 \setminus \supp \ S)$, then
$lk(\gamma_1,S) = lk(\gamma_2,S)$.
\end{prop}

Moreover,
since the pairing $\left<\cdot,S\right>$ is linear in the space
of chains $\sigma$ (having $\partial \sigma$ disjoint from $\supp \ S$), the linking number descends to a homomorphism:
\begin{eqnarray*}
lk(\cdot,S): H_1(\CP^2 \setminus \supp \ S) \rightarrow \RR/\ZZ.
\end{eqnarray*}
Similarly $lk(\cdot,S) :H_1(\Omega) \rightarrow \RR/\ZZ$ for any open $\Omega \subset \CP^2 \setminus \supp \ S$.

\begin{thm}\label{THM:GENERAL_TECHNIQUE}
Suppose that $f:\CP^2 \rightarrow \CP^2$ is a holomorphic endomorphism and $\Omega
\subset U(f)$ is contained in a union of basins of attraction of attracting periodic points
for $f$.  If there are $c \in H_1(\Omega)$ with linking number $lk(c,T) \neq 0$
arbitrarily close to $0$ in $\QQ/\ZZ$, then $H_1(\Omega)$ is infinitely
generated.
\end{thm}

%========================================
\section{Proof of the Main result} 
\label{SEC:MAIN}
%========================================

We recall the characterization of Axiom-A polynomial skew products from \cite[\S 8]{JON_SKEW}.
(For the actual definition of Axiom-A, see, for example, \cite[Def. 8.1]{JON_SKEW}.)
Let $A_p$ be the set of attracting periodic points of $p$ and
consider the following postcritical sets:
\begin{eqnarray*}
D_p &:=& \overline{\bigcup_{n \geq 1} p^n C_p}, \,\, \mbox{where} \,\, C_p = \{z \, : \, p'(z) = 0\}, \\
D_{J_p} &:=& \overline{\bigcup_{n \geq 1} f^n C_{J_p}}, \,\, \mbox{where} \,\, C_{J_p} := \{(z,w) \, : \, z \in J_p, q_z'(w) = 0\}, \\
D_{A_p} &:=& \overline{\bigcup_{n \geq 1} f^n C_{A_p}}, \,\, \mbox{where} \,\, C_{A_p} := \{(z,w) \, : z \in A_p, q_z'(w) = 0\}, \,\, \mbox{and} \\
D_\Pi &:=& \overline{\bigcup_{n \geq 1} f_\Pi^n C_\Pi},  \,\, \mbox{where} \,\, C_{\Pi} := \{\lambda \, : f_\Pi'(\lambda) = 0\}.
\end{eqnarray*}

Let 
\begin{eqnarray*}
J_{A_p} = \overline{\bigcup_{z \in A_p} \{ z\} \times J_z}.
\end{eqnarray*}
The following appears as \cite[Cor. 8.3]{JON_SKEW}:
\begin{prop}\label{PROP:AXA}
A polynomial skew product $f=(p,q)$ is Axiom-A on $\CP^2$ if and only if:
\begin{enumerate}
\item $D_p \cap J_p = \emptyset$,
\item $D_{J_p} \cap J_2 = \emptyset$,
\item $D_{A_p} \cap J_{A_p} = \emptyset$, and
\item $D_\Pi \cap J_\Pi = \emptyset$.
\end{enumerate}
\end{prop}

In particular, if $f$ is Axiom-A, then
\begin{itemize}
\item every Fatou component of $p$ (respectively $f_\Pi$) is the basin of attraction of an attracting
cycle whose immediate basin contains a critical point from $C_p$ (respectively $C_\Pi$), and
\item $C_z \cap J_z = \emptyset$ over every $z\in A_p \cup J_p$.
\end{itemize}

\vspace{0.05in}

We can describe the fiberwise dynamics as follows: Let $L := \{W=0\}$ be the
horizontal projective line.  The vertical projection $\pi(z,w) = (z,0)$
induces a rational map $\pi: \CP^2 \ra L$ whose only point of indeterminacy is
$[0:1:0]$ (which blows up under $\pi$ to the entire line $L$).  Although $L$
is not invariant under $f$, we can consider the action of $p$ on the line $L$
and we let $\mu_p$ be the harmonic measure on $J_p$, i.e. $\mu_p := dd^c
G_p(z)$.  Within $\CP^2$ we have the vertical current $T_p := \pi^* \mu_p$.

Throughout the remainder of the paper we will denote by $C$ the closure in
$\CP^2$ of the ``vertical critical points'' $\{(z,w) \,:\, q_z'(w) = 0\}$.
It is an algebraic curve of degree $d-1$ 
with $C_z = C \cap (z \times \CC)$ and  $C \cap \Pi = C_\Pi$.  Since $q_z(w) = w^d + O_z(w)$, we have that
$[0:1:0] \not \in C$.  Therefore, $\pi: C \ra L$ is a branched covering of degree 
$\deg C = d-1$.

\begin{prop}\label{PROP:SETUP1}
If $f$ is a polynomial skew product with $J_z$ connected for every $z \in J_p$, then: 
\begin{enumerate}
\item For any piecewise-smooth two-chain $\Gamma \subset C$, we have $\left<\Gamma,T \right> = \left<\Gamma,T_p\right>$.
\item Let $U$ be any Fatou component $U$ of $p$ and $C_U: = \pi_{|C}^{-1}(U)$.  The sequence of restrictions $\{f^n_{|C_U}\}$ is a 
normal family\footnote{In other words,
$\{f^n \circ \rho\}$ is a normal family for any parameterization $\rho$ of $C_U$ (see below).}.
\end{enumerate}
\end{prop}

\begin{proof}

It suffices to consider any irreducible component $\C$ of $C$, which we parameterize
by its normalization $\rho:
\widehat \C \ra \C$; see \cite{GUNNING}.  
Since $\C$ has dimension $1$, $\widehat \C$ is a compact Riemann surface.
The following diagram summarizes the maps $\rho$ and $\pi$:
\begin{diagram}
\widehat \C & \rTo^{\rho} & \C \\
 & \rdTo_{\pi \circ \rho}   & \dTo_{\pi}\\
 &  & L.\\
\end{diagram}

Since $[0:1:0] \not \in \C$, we need only consider $\C$ within two systems of
affine coordinates: the original affine coordinates $(z,w)$ and the coordinates 
$t=T/Z$ and $u = W/Z$, defined in a neighborhood $\Omega$ of $[1:0:0]$.  
In the $(z,w)$ coordinates the Green's current is given by $T =
\frac{1}{2\pi}dd^c G(z,w)$ with
\begin{eqnarray}\label{EQN:GREEN_DECOMP1}
G(z,w) = G_z(w) + G_p(z).
\end{eqnarray}
Meanwhile, in the $(t,u)$ coordinates it is given by $T = \frac{1}{2\pi}dd^c G_\Omega(t,u)$ with
\begin{eqnarray}\label{EQN:GREEN_DECOMP2}
G_\Omega(t,u) = G_t(u) + G_p^\#(t).
\end{eqnarray}
\noindent
Here, $G_p^\#(t)$ is obtained by extending $G_p(1/t) - \log(1/t)$ continuously
through $t = 0$ and $G_t(u)$ is the PSH extension of $G_z(w)$ described in
\cite[Lemma 6.3]{JON_SKEW}.  One has that $(0,u) \in K_\Pi$ if and only if
$G_0(u) = 0$ and, for $t \neq 0$, $(t,u) \in K_{1/t}$ if and only if $G_t(u) =
0$.

Because $J_z$ is connected for every $z \in J_p$, Proposition 6.4 from \cite{JON_SKEW}
gives that $C_z \subset K_z$ for
every $z \in \CC$ and $C_\Pi \subset K_\Pi$.  Therefore, if $(z,w) \in \C \cap
\CC^2$ we have that $G_z(w) = 0$ and (\ref{EQN:GREEN_DECOMP1}) gives that the
restriction of $G$ to $\C \cap \CC^2$ coincides with $G_p(z)$.  Similarly, using
(\ref{EQN:GREEN_DECOMP2}), the restriction of $G_\Omega$ to $\C
\cap \Omega$ is just $G_p^\#(t)$.  

The above calculations give 
\begin{eqnarray}\label{EQN:RESTRICTIONS}
\rho^*  \ T = \rho^* \ T_p = \rho^* \pi^* \mu_p.
\end{eqnarray}
\noindent
Let $\Gamma \subset \C$ be any piecewise smooth two chain and $\widehat \Gamma
\subset \widehat \C$ its lift to the normalization.  Then, 
\begin{eqnarray*}
\left<\Gamma,T\right> = \left<\widehat \Gamma,\rho^* T\right> = \left<\widehat \Gamma,\rho^* T_p\right> =  \left<\Gamma,T_p\right>.
\end{eqnarray*}

There is a general principle that for any complex manifold $M$ and holomorphic map $\phi:M \ra \CP^2$, 
the family $\{f^n \circ \phi\}$ is normal in a neighborhood of $x \in M$ if and only if $x \in M \sm \supp \phi^* T$; See \cite{FS2} and
also \cite[p. 409]{JON_SKEW}.

In our situation, it follows that
\begin{eqnarray}\label{EQN:NORMAL_SEQ}
f^n \circ \rho: \widehat \C \ra \CP^2
\end{eqnarray}
\noindent
is normal at $x \in \widehat \C$ if and only if $x \in \widehat \C \sm
\supp(\rho^* \ T_p)$.    Since $G_p$ is harmonic outside of $J_p$, we have that
$\C_U$ is disjoint from $\supp \ T_p$ and, hence, that $\hat \C_U$ is disjoint from
$\supp (\rho^* \ T_p)$.  Therefore, $f^n \circ \rho$ is normal on $\hat \C_U$.
\end{proof}

\begin{cor}\label{COR:SETUP2}
Suppose that $f$ is Axiom-A and satisfies the hypotheses of Proposition
\ref{PROP:SETUP1}, $z_0$ is an attracting periodic point for $p$ (possibly $z_0
= \infty$), and $U$ is any component of the immediate basin for $z_0$.  Then,
$C_U$ is in the immediate basins for some attracting periodic points of $f$
within $z=z_0$. 
\end{cor}

\begin{proof}
Since $f$ is Axiom-A, $\C_{z_0} \subset \C_U$ is in the immediate
basins of attraction for some attracting periodic points of $f_{|\{z=z_0\}}$.
Because the line $z=z_0$ is transversely attracting, these periodic points 
are also attracting for $f$ in $\CP^2$, giving an open subset of $\C_U$
is in these immediate basins. Since
$\{f^n \circ \rho\}$ is normal on $\rho^{-1}(\C_U)$, all of
$\C_U$ is in the union of these immediate basins.
\end{proof}

\vspace{0.2in}

\noindent
{\bf Proof of Theorem \ref{THM}:}\\
\indent
{\em Case 1: $f$ is connected:}\\
\indent
We begin with Fatou components for $f$ that are bounded in $\CC^2$, i.e. those on which $G(z,w) \equiv 0$.

Since $J_p$ is connected and $p$ is hyperbolic, the Fatou set $U(p)$ consists
of the basins of attraction of finitely many attracting periodic orbits
together with the basin of attraction for the superattracting fixed point at
$z=\infty$.  Moreover, each component of these basins is conformally equivalent
to the open unit disc $\DD$.  For simplicity of exposition, we suppose that each attracting periodic orbits for $p$ is
a fixed point (otherwise, one can pass to an appropriate iterate of $f$).

Let $z_0 \in \CC$ be an attracting fixed point of $p$ and let $U_0$ be the
component of its basin of attraction that contains $z_0$. 

Let $x$ and $y$ denote points in $\CC^2$.

Since $f$ is Axiom-A, $J_{z_0}$ is a hyperbolic set, having a local stable
manifold $W^s_\loc(J_{z_0})$ formed as the union of stable curves
$W^s_\loc(x)$ of points $x \in J_{z_0}$ with each $W^s_\loc(x)$ 
being holomorphic.   These stable curves satisfy the invariance
\begin{eqnarray}\label{EQN:INVARIANCE}
f(W^s_\loc(x)) \subset W^s_\loc(f(x)).
\end{eqnarray}
\noindent
They are transverse to the vertical line 
$\{z=z_0\}$ since it is the unstable direction of $J_{z_0}$.
Hence, there is a sufficiently small $\eps > 0$ so that each
$W^s_\loc(x)$ is the graph of a holomorphic function over $z \in
\DD_\eps(z_0) := \{|z - z_0| < \eps\}$.
In other words, this describes $W^s_\loc(J_{z_{0}})$ as a holomorphic motion
(see, for example, \cite[\S 5.2]{HUBB}) over the open disc $z \in
\DD_\eps(z_{0})$.   
%(We will assume that $\eps$ is sufficiently small so that $p(\DD_\eps(z_{0})) \subset \DD_\eps(z_{0})$.)

Because $f|_{z=z_0}: \{z_0\} \times \CC \ra \{z_0\} \times \CC$ is hyperbolic,
its Fatou set consists of the basins of attraction of finitely many attracting
periodic points (including $w=\infty$).  We will use some hyperbolic theory to
show that if $z \in \DD_\eps(z_0)$,  then either $(z,w) \in W^s_\loc(J_{z_0})$ or $(z,w)$ is in the basin of one of
these finitely many attracting periodic points for $f|_{z=z_0}$.  Because most
classical treatments of hyperbolic theory are for diffeomorphisms, we appeal
to the discussion for endomorphisms that appears in \cite{JON_THESE}.

The natural extension of $J_{z_0}$ is
\begin{eqnarray*}
\widehat J_{z_0} = \{(x_i)_{i \leq 0} \, :  x_i \in J_{z_0} \, \mbox{and} \, f(x_i) = x_{i+1}\}.
\end{eqnarray*}
Associated to any prehistory $\hat x \in \widehat J_{z_0}$ is a local unstable manifold $W^u_\loc(\hat x)$,
which, in this case, is just an open disc in the vertical line $z=z_0$.

We will now verify that $\widehat J_{z_0}$ has a {\em local product structure};
see Definition 2.2 from \cite{JON_THESE}.  If $W^s_\loc(x) \cap W^u_\loc(\hat
y)$ is non-empty for some $x \in J_{z_0}$ and $\hat y \in \widehat J_{z_0}$,
then $W^s_\loc(x) \cap W^u_\loc(\hat y) = x$, since $W^u_\loc(\hat y)$ is
contained in the vertical line $\{z=z_0\}$ and $W^s_\loc(x)$ is the graph of a
function of $z$.  In particular, the intersection is a unique point.  Moreover,
the unique prehistory $\hat x$ of $x$ satisfying that $x_i \in W^u_\loc(f^i
(\hat y))$ for all $i \leq 0$ is in $\widehat J_{z_0}$, since every local
unstable manifold lies in $\{z=z_0\}$.

Because $\widehat J_{z_0}$ has a local product structure, Corollary 2.6 from
\cite{JON_THESE} gives that there is a neighborhood $V$ of $J_{z_0}$ in $\CC^2$
so that if $f^n(z,w) \in V$ for all $n \geq 0$, then $(z,w) \in W^s_\loc(x)$
for some $x \in J_{z_0}$.  
If $z \in \DD_\eps(z_0)$, then $f^n(z,w)$ must converge to the
line $\{z=z_0\}$ since $\DD_\eps(z_0) \subset U_0$.  Existence of the
neighborhood $V$ implies that $(z,w)$ is either in $W^s_\loc(J_{z_0})$
or is in the basin of attraction for one of the finitely many attracting
periodic points for $f|_{z=z_0}: \{z_0\} \times \CC \ra \{z_0\} \times \CC$.

\vspace{.1in}
We can now see that if $z \in \DD_\eps(z_0)$, then $J_z = W^s(J_{z_0}) \cap
(\{z\} \times \CC)$.  From the preceding paragraph we see that if $(z,w) \not
\in W^s_\loc(J_{z_0})$, then $w \not \in J_z$.  Meanwhile, if $(z,w) \in
W^s_\loc(J_{z_0})$, then the family of iterates $f^n|_{\{z\} \times \CC}$
cannot be normal at $w$, giving that $w \in J_z$. 

\vspace{.1in}

Repeatedly taking the preimages of $W^s_\loc(J_{z_{0}})$ under $f$ that
intersect the vertical line $\{z=z_{0}\}$, we obtain a stable set
$W^s(J_{z_0})$ over all of $U_0$.  
For any $z \in U_0$, we have that $(z,w)$ is either in $W^s(J_{z_0})$ or is in the basin attraction
for one of the finitely many attracting periodic points of $f|_{z=z_0}$.
We also have 
$J_z = W^s(J_{z_0}) \cap (\{z\} \times \CC)$.

By Corollary \ref{COR:SETUP2}, $C_{U_0} \subset U(f)$,
 so that $C_{U_0}$ is disjoint from $W^s(J_{z_{0}})$. 
A repeated application of the Inverse
Function Theorem, together with (\ref{EQN:INVARIANCE}), allows us to extend (in
the parameter $z$) the holomorphic motion of $J_{z_{0}}$ from $z \in
\DD_\eps(z_{0})$ to a holomorphic motion of $J_{z_{0}}$ over all of $U_{0}$.  
(For details, see the proof of Theorem 6.3 from \cite{HRUSKA_ROEDER}.)

The image of this holomorphic motion is precisely $W^s(J_{z_{0}})$.  By
Slodkowski's Theorem \cite{SLOD}, it will extend (in the
fiber $w$) as a holomorphic motion of the entire Riemann Sphere $\CP^1$ that is 
parameterized by $U_{0}$.  Therefore, every bounded Fatou
component of $f$ lying above $U_{0}$ is given as the holomorphic motion of a Fatou
component of $f|_{z=z_{0}}$.  Each of these is a disc (since
$f|_{z=z_{0}}$ is hyperbolic with connected Julia set), so  every bounded 
Fatou component lying above $U_0$ is homeomorphic to an open ball.

\vspace{0.1in}

Let $z_n$ be an $n$-th preimage under $p$ of $z_0$ and $U_n$ be the
component of $W^s(z_0)$ containing $z_n$.  We will use induction on $n$ to show that there is a holomorphic motion of
$J_{z_n}$ parameterized by $z \in U_n$ so that
\begin{enumerate}
\item  for every $z \in U_n$,  $J_z$ is obtained as the motion of $J_{z_n}$, and
\item the motion of any $w \in J_{z_n}$ is precisely stable curve of $w$.
\end{enumerate}
\noindent
Then, as above, Slodkowski's Theorem will give that every bounded
Fatou component of $f$ lying over $U_n$ is obtained as the holomorphic motion
of some component of $K_{z_n} \sm J_{z_n}$, parameterized by $z \in U_n$.  In
particular, every such Fatou component will be homeomorphic to an open ball.

The desired holomorphic motion already exists for $n = 0$, so we suppose that
it exists for $n=k$ in order to prove it for $n=k+1$.  

It is sufficient to show that the part $C_{U_{k+1}}$ of the horizontal critical locus
that lies above $U_{k+1}$ is in the Fatou set $U(f)$.  In that
case, the Implicit Function Theorem can be used to lift the entire holomorphic
motion of $J_{z_k}$ under $f$ to a holomorphic motion of $J_{z_{k+1}}$
parameterized by $U_{k+1}$.  The result will automatically satisfy (1) and (2).

Let $\C_{U_{k+1}}$ be some irreducible component of $C_{U_{k+1}}$.  
By Proposition \ref{PROP:SETUP1}, the family of restrictions of iterates
$\{f^n_{|\C_{U_{k+1}}}\}$ is normal.  Therefore, 
$f^{k+1} \C_{U_{k+1}}$ is either in the Fatou set $U(f)$ or is within the stable
curve of a single $x \in J_{z_0}$.   

Suppose the latter happens.  Then, every point of
$f^{n+1} \C_{U_{k+1}}$ lying over any $z \in U_0$ is in~$J_z$.  Let $\C_{\partial U_{k+1}}$ be
portion of the horizontal critical locus lying over $\partial U_{k+1} \subset
J_p$ that is in the closure of $\C_{U_{k+1}}$.  Lower semi-continuity 
of the mapping $z \mapsto
J_z$ (see  \cite[Prop. 2.1]{JON_SKEW}), combined with continuity of $f^{k+1}$, gives that
\begin{eqnarray*}
f^{k+1} \C_{\partial U_{k+1}} \subset \bigcup_{z \in \partial U_0} \{z\} \times J_z \subset \overline{\bigcup_{z \in J_p} \{z\} \times J_z} = J_2.
\end{eqnarray*}
However, since $f$ is Axiom-A, 
Proposition \ref{PROP:AXA} implies that $f^{k+1}
\C_{\partial U_{k+1}} \subset D_{J_p}$ is disjoint from $J_2$.  Therefore, we conclude that
every component of $C_{U_{k+1}}$ is in the Fatou set.

\vspace{.1in}

We now consider the Fatou components for $f$ on which $G(z,w) > 0$.  
Recall that $(z,w) \in W^s([0:1:0])$ if and only if $G_z(w) > 0$.
Since $J_z$ is connected for every $z \in J_p$, it follows from Theorem
\ref{THM:HR} that $W^s([0:1:0])$ is homeomorphic to an open ball.

It remains to consider Fatou components on which $G_p(z) > 0$ and $G_z(w) = 0$.
Let $U_\infty \subset \CP^1$ be the basin of attraction of $\infty$ under $p$.
Since $J_p$ is connected, $U_\infty$ is simply connected.  For simplicity of
exposition, we suppose that $0 \not \in U_\infty$ (otherwise, one can conjugate
$f$ by an appropriate translation in the $z$-coordinate).

Let us work in the system of local coordinates $t=T/Z$ and $u = W/Z$ in which
\begin{eqnarray*}
f(t,u) = \left(\frac{t^d}{P(1,t)}, \frac{Q(1,u,t)}{P(1,t)}\right) \equiv (r(t),s_t(u)),
\end{eqnarray*}
which is a rational skew product.  The first coordinate has $0$ as a
totally-invariant superattracting fixed point, whose basin of attraction 
$U_\infty$ is entirely contained in the copy of $\CC$ parameterized
by $t$.  The Fatou components of $f$ that remain to be studied each have
projection under $(t,u) \mapsto t$ lying entirely in $U_\infty$.  

The line $\Pi$ is given by $t=0$.  Since $f$ is Axiom-A, $J_\Pi$ is hyperbolic,
having a local stable manifold $W^s_\loc(J_\Pi)$ that is formed as the union of
stable curves of points $x \in J_\Pi$.  Each of these stable curves is
transverse to $\Pi$, since $\Pi$ is the unstable direction of $J_\Pi$.  Thus,
we can choose $\eps > 0$ sufficiently small so that $W^s_\loc(J_\Pi)$ is
described as a holomorphic motion of $J_\Pi$ that is parameterized by $t \in
\DD_\eps(0)$.

Using the same proof as for $\widehat J_{z_0}$, one can check that $\widehat J_\Pi$ has a local
product structure, so that if $(t,u)$ satisfies $|t| < \eps$, then $(t,u)$ is
in the basin of attraction for one of the finitely many attracting periodic
points of $f_\Pi$ (including $u = \infty$, which corresponds to $[0:1:0]$).
Similarly, if $|t| < \eps$, then
$J(f) \cap (\CC \times \{t\}) = W^s(J_\Pi) \cap (\CC \times \{t\})$.

By Corollary \ref{COR:SETUP2}, the horizontal critical locus
\begin{eqnarray*}
C_{U_\infty} = \left\{(t,u) \,\, :\,\, \frac{\partial s_t(u)}{\partial u} = 0 \,\, \mbox{and} \,\, t \in U_\infty\right\}
\end{eqnarray*}
lies in the Fatou set for $f$.  As before, the Implicit Function
Theorem can be used repeatedly to extend the holomorphic motion of $J_\Pi$ to
one that is parameterized by all of $U_\infty$.  Then, Slodkowski's Theorem can be used
to extend this holomorphic motion in the fiber $w$.  Each of the Fatou
components of $f$ that lies over $U_\infty$ (other than $W^s([0:1:0])$) is then a
holomorphic motion of a bounded Fatou component of the polynomial map $f_\Pi$,  parameterized by the simply connected domain
$U_\infty$.  This gives that each such component is homeomorphic to
an open bidisc.

\vspace{0.1in}
\indent
{\em Case 2: $f$ is disconnected:}\\
\indent
If $J_z$ is disconnected for any $z \in J_p$, then Theorem \ref{THM:HR} gives that
$H_1(W^s([0:1:0]))$ is infinitely generated.  It remains to consider the case that
$J_z$ is connected for every $z \in J_p$ and $J_p$ is disconnected.

Let $U_\infty$ be the basin of attraction of $z = \infty$ for $p(z)$.   Corollary
\ref{COR:SETUP2} gives that each of the irreducible components of $C_{U_\infty} = \pi_{|C}^{-1}(U_\infty)$
is in the immediate basin of attraction $W^s_0(\zeta_i)$ of one of the finitely
many attracting periodic points
\begin{eqnarray*}
\zeta_1 := \{\zeta_1^1,\ldots,\zeta_1^{n_1}\},\ldots, \zeta_m := \{\zeta_m^1,\ldots,\zeta_m^{n_m}\}
\end{eqnarray*}
\noindent
of $f_\Pi$.

The proof of Proposition 4.2 from \cite{HRUSKA_ROEDER} shows how to generate a
sequence of piecewise smooth one-cycles 
$\upsilon_n \subset U_\infty$ bounding regions $\Upsilon_n \subset \CC$ with the
property that $\left<\Upsilon_n,\mu_p\right> \rightarrow 0$.
Perturbing the $\upsilon_i$ slightly (if needed), we can suppose that
none of them lie on the finitely many critical values of $\pi: C \ra L$.
For each $i$ we let 
\begin{eqnarray*}
\gamma_i := \pi^{-1}(\upsilon_i) \,\,\, \mbox{and} \,\,\, \Gamma_i := \pi^{-1}(\Upsilon_i),
\end{eqnarray*}
so that each $\gamma_i$ is a finite union of closed loops in $C_{U_\infty}$ bounded by
a piecewise smooth chain $\Gamma_i$.  Moreover, $\pi: \Gamma_i \ra \Upsilon_i$ is a ramified cover of degree $d-1$.

Proposition \ref{PROP:SETUP1} gives $\left<\Gamma_i,T \right>  =
\left<\Gamma_i, T_p\right>$ and the
invariance properties of $\left<\cdot,\cdot \right>$ from Proposition \ref{PROP:PAIRING_INV} give
\begin{eqnarray*}
\left<\Gamma_i, T_p\right> = \left<\Gamma_i, \pi^*
\mu_p\right> = \left<\pi_* \Gamma_i , \mu_p\right> = (d-1)
\left< \Upsilon_i,\mu_p\right>.
\end{eqnarray*}
Since $\left< \Upsilon_i,\mu_p\right> \ra 0$, we have that
\begin{eqnarray*}
lk(\gamma_i,T) = \left<\Gamma_i,T \right> \ (\text{mod} \ 1)  = (d-1) \left< \Upsilon_i,\mu_p\right> \ (\text{mod} \ 1) \ra 0 \ (\text{mod} \ 1).
\end{eqnarray*}
Proposition \ref{THM:GENERAL_TECHNIQUE} gives that the first homology of
$\bigcup_{i} W^s_0(\zeta_i)$ is infinitely generated.  Hence at least one of
the finitely many Fatou components from this union has infinitely generated
homology.
\vspace{0.1in}
\noindent
$\Box$

\begin{rmk}\label{RMK:HYPERBOLICITY}
We will now discuss the question of whether Theorem \ref{THM} holds without the hypothesis that $f$ be Axiom-A.
To see that some condition on hyperbolicity is needed, consider the {\em disconnected} product mapping
\begin{eqnarray*}
f(z,w) = (z^2-6,w^2+i),
\end{eqnarray*}
which has that $J_p$ is a Cantor set and that $J_z$ is a (connected) dendrite for
every $z \in \CC$.   Since $J_z$ is connected for every $z \in J_p$, Theorem
\ref{THM:HR} gives that $W^s([0:1:0])$ is homeomorphic to an open ball.
Moreover, the closure of $W^s([0:1:0])$ is all of $\CP^2$, giving that it is the only Fatou component for $f$. 

Meanwhile, the author does not presently know of any example of a connected polynomial skew product having any Fatou component that is not homeomorphic to an open ball.

\vspace{0.05in}
One can somewhat weaken the hypothesis that $f$ be Axiom-A and still have the proof of Theorem \ref{THM} that is presented in this section hold.  
(We leave the Axiom-A condition in the statement of Theorem \ref{THM}, in order to keep it simple.)
More specifically:

\vspace{0.05in}
If $J_z$ is disconnected for some $z \in J_p$, then no hypothesis is needed in order to determine that $W^s([0:1:0])$ has infinitely generated first homology.
See Example \ref{EG:NON_AXA}.

\vspace{0.05in}
If $J_z$ is connected for every $z \in J_p$, but $J_p$ is disconnected, we only
need $f_\Pi$ has at least one attracting periodic orbit $\zeta_0$.  Then, some
some irreducible component of $C_{U_\infty}$ will be in  $W^s_0(\zeta_i)$ and the
technique from Case 2 can be applied to that component.

\vspace{0.05in}
If $f$ is connected, then one need only assume that $J_p, J_\Pi,$ and $J_{A_p}$
be hyperbolic.   In this case, a slight modification of the proof of Theorem
\ref{THM} is needed to address the issue of bounded Fatou components lying over
a preimage $U_{k+1}$ of some immediate basin $U_0$ for $p$.    Without assuming
that $f$ is Axiom-A, $f^{k+1} \C_{U_{k+1}}$ could potentially land in the
stable curve $W^s(x)$ of some $x \in J_{z_0}$.  However, one can still use the
Implicit Function Theorem to pull back all but finitely many leaves from the
holomorphic motion $W^s(J_{z_k})$.   By the Lambda Lemma
\cite{LYU_LAMBDA,MSS}, this holomorphic motion (that is missing a few leaves) can be extended to a holomorphic motion of all of 
$J_{z_{k+1}}$.  One then proceeds using Slodkowski's Theorem, as in the original proof.

\end{rmk}

\section{Examples}

We conclude this note by interpreting some of the examples of Axiom-A polynomial skew products from
\cite{JON_SKEW}, \cite{S}, and \cite{SUMI_III} in the context of Theorem \ref{THM}.

\begin{example}
Trivial examples of Axiom-A polynomial skew products exhibiting both behaviors from Theorem 
\ref{THM}
come from product maps $f(z,w) = (p(z),q(w))$ with
$p$ and $q$ hyperbolic and small perturbations thereof.
\end{example}

\begin{example}
The family 
\begin{eqnarray*}
f_a(z,w) = (z^2,w^2+az).
\end{eqnarray*}
is considered by DeMarco-Hruska in Section 5.3 from \cite{S}.
The map $f_a$ is Axiom-A if and only if $g_a(w) := w^2+a$ is hyperbolic \cite[Theorem
5.1]{S}.   (In that case, $f_a$ is in the same hyperbolic component as a product if and only if $g_a$ has an attracting fixed point.)
Moreover, $J_p$ is the unit circle $\{|z| = 1\}$ and $J_{e^{it}}$
is a rotation of angle $t/2$ of $J_{1} = J(g_a)$; see \cite[Lemma 5.5]{S}.
Thus, $f_a$ is connected if and only if $a$ is in the Mandelbrot set
$\mathcal{M}$.   

If $a$ is a hyperbolic point from $\mathcal{M}$, then 
Theorem \ref{THM} gives that the Fatou set of $f_a$ is a union of open balls.  Meanwhile, if $a \in \CC \sm \mathcal{M}$, then the basin of attraction $W^s([0:1:0])$ has infinitely generated first homology. 
(See also \cite[Theorem 6.3]{HRUSKA_ROEDER}, where the assumption that $f_a$ be Axiom-A is dropped.)
\end{example}

\begin{example}
Jonsson shows in
 Example 9.6 from \cite{JON_SKEW} that
\begin{eqnarray*}
f(z,w) = (z^2-6,w^2+3-z)
\end{eqnarray*}
is Axiom-A on $\CP^2$ and is not in the same hyperbolic component as any product mapping.  
There are two reasons why $f$ is disconnected:
\begin{enumerate}
\item $p(z) = z^2-6$ has a disconnected Julia set, and
\item $J_{z_0}$ is disconnected over the repelling fixed point $z_0 = -2 \in J_p$.
\end{enumerate}
Since $J_{z_0}$ is disconnected, Theorem \ref{THM:HR} gives that $W^s([0:1:0])$ has infinitely generated first homology.

The action on the line at infinity is given by $u \mapsto u^2$, where $u =
w/z$, so that $f$ has two superattracting fixed on that line: $[0:1:0]$ and
$[1:0:0]$.  Even though $J_p$ is disconnected, the proof of Case 2 from Theorem
\ref{THM} does not give that $W^s([1:0:0])$ has infinitely generated
first homology, because some points on the horizontal critical locus $C = \{w =
0\}$ have orbits escaping to $[0:1:0]$.  (We do not know if $W^s([1:0:0])$ has complicated first homology.)
\end{example}

\begin{example}
The following example by Sumi appears in \cite[Remark 4.13]{SUMI_III}.
Let
\begin{eqnarray*}
f(z,w) = \left(p(z), w^{2^n} + \left(\frac{z+\sqrt{R}}{2\sqrt{R}}\right) t_{n,\eps}(w)\right),
\end{eqnarray*}
where $R, \eps > 0$, $n \in \NN$, $p_R(z) = z^2-R$, $p = p_R^n$, $h_\eps(w) =
(w-\eps)^2 -1 + \eps$, and $t_{n,\eps}(w) = h^n_\eps(w) - w^{2^n}$.  For
appropriate choices of $\eps$ sufficiently small, $R$ sufficiently large, and
$n$ sufficiently large and even, one has that $f$ is Axiom-A, not in the same
hyperbolic component as any product mapping, with $J_p$ a Cantor set and
$J_z$ connected over every $z \in \CC$.

The action of $f$ on the line at infinity is given by $u \mapsto u^{2^n}$, thus
$f$ has two superattracting fixed points $[0:1:0]$ and $[1:0:0]$ on the line at
infinity.  Since $J_z$ is connected for every $z \in \CC$ and $J_p$ is
disconnected, the proof of Theorem \ref{THM} gives that $W^s([1:0:0])$ has
infinitely generated first homology.  Meanwhile, Theorem \ref{THM:HR} gives that
$W^s([0:1:0])$ is homeomorphic to an open ball.

\end{example}

\begin{example}\label{EG:NON_AXA}
Jonsson shows in Example 9.7 from \cite{JON_SKEW} that
\begin{eqnarray*}
f(z,w) = (z^2-2,w^2+2(2-z))
\end{eqnarray*}
has $J_2$ connected but is not a connected skew product
according to Definition \ref{DEFN} since $J_z$ is disconnected for $z=-2$.  
This skew product is not Axiom-A because $p(z)$ has its critical point in the Julia set $J_p = [-2,2]$.
However, Theorem \ref{THM:HR} can still be applied 
to see that $W^s([0:1:0])$ has infinitely generated first homology.
\end{example}

\bibliographystyle{plain}
\bibliography{newton.bib}

\end{document}